\documentclass[12pt]{amsart}
\usepackage{amsfonts,amsthm,latexsym,amsmath,amssymb,amscd,amsmath, mathrsfs,
url}
\newcounter{count}

\numberwithin{count}{section}
\newtheorem{Lemma}[count]{Lemma}

\newtheorem{Definition}[count]{Definition}

\newtheorem{Theorem}[count]{Theorem}

\begin{document}

\author[T. H.~Nguyen]{Thu Hien Nguyen}

\address{Department of Mathematics \& Computer Sciences, V. N. Karazin Kharkiv National University,
4 Svobody Sq., Kharkiv, 61022, Ukraine}
\email{nguyen.hisha@gmail.com }

\author[A.~Vishnyakova]{Anna Vishnyakova}
\address{Department of Mathematics \& Computer Sciences, V. N. Karazin Kharkiv National University,
4 Svobody Sq., Kharkiv, 61022, Ukraine}
\email{anna.m.vishnyakova@univer.kharkov.ua}

\title[Entire functions of the Laguerre-P\'olya class]
{On the necessary condition for entire function with the increasing second quotients of 
Taylor coefficients to belong to the Laguerre-P\'olya class }

\begin{abstract}

For an entire function $f(z) = \sum_{k=0}^\infty a_k z^k, a_k>0,$  we show that $f$
 does not belong to the Laguerre-P\'olya class  if the quotients $\frac{a_{n-1}^2}{a_{n-2}a_n}$ are increasing in $n$,
and  $c:= \lim\limits_{n\to \infty} \frac{a_{n-1}^2}{a_{n-2}a_n}$ is smaller than an absolute constant $q_\infty$   
$(q_\infty\approx 3{.}2336) .$ 

\end{abstract}



\maketitle

\section{Introduction}

The zero distribution of  entire functions, its sections and tails have been 
studied by many authors, see, for example, the remarkable
survey of the topic in \cite{iv}.  In this paper we investigate a new necessary 
condition under which some special entire functions have only real zeros. 
First, we need the definition of the 
famous Laguerre-P\'olya class.

{\bf Definition 1}.  A real entire function $f$ is said to be in the {\it
Laguerre-P\'olya class}, written $f \in \mathcal{L-P}$, if it can
be expressed in the form
\begin{equation}
\label{lpc}
 f(x) = c x^n e^{-\alpha x^2+\beta x}\prod_{k=1}^\infty
\left(1-\frac {x}{x_k} \right)e^{xx_k^{-1}},
\end{equation}
where $c, \alpha, \beta, x_k \in  \mathbb{R}$, $x_k\ne 0$,  $\alpha \ge 0$,
$n$ is a nonnegative integer and $\sum_{k=1}^\infty x_k^{-2} <
\infty$. As usual, the product on the right-hand side can be
finite or empty (in the latter case the product equals 1).

This class is essential in the theory of entire functions due to the fact that  the polynomials 
with only real zeros converge locally uniformly to these and only these 
functions. The following  prominent theorem 
states an even  stronger fact. 

{\bf Theorem A} (E.Laguerre and G.P\'{o}lya, see, for example,
\cite[p. 42-46]{HW}). {\it   

(i) Let $(P_n)_{n=1}^{\infty},\  P_n(0)=1, $ be a sequence
of complex polynomials having only real zeros which  converges uniformly in
the circle $|z|\le A, A > 0.$ Then this sequence converges
locally uniformly to an entire function
 from the $\mathcal{L-P}$ class.

(ii) For any $f \in \mathcal{L-P}$ there is a
sequence of complex polynomials with only real zeros which
converges locally uniformly to $f$.}

For various properties and characterizations of the
Laguerre-P\'olya class see 
\cite[p. 100]{pol}, \cite{polsch}  or  \cite[Kapitel II]{O}.

Note that for a real entire function (not identically zero) of order less than $2$ having 
only real zeros is  equivalent to belonging to the Laguerre-P\'olya class. The situation is 
different when an entire function is of order $2$. For example, the function 
$f_1(x)= e^{- x^2}$ belongs to the Laguerre-P\'olya class, but the function 
$f_2(x)= e^{x^2}$ does not.

Let  $f(z) = \sum_{k=0}^\infty a_k z^k$  be an entire function with positive coefficients. 
We define the quotients $p_n$ and $q_n$:

\begin{eqnarray}
\label{qqq} &  p_n=p_n (f):=\frac {a_{n-1}}{a_n},\ n\geq 1;
\\
\nonumber &
q_n=q_n(f):=\frac {p_n}{p_{n-1}}=\frac {a_{n-1}^2}{a_{n-2}a_n},\
n\geq 2.
\end{eqnarray}

The following formulas can be verified by straight forward calculations.

\begin{eqnarray}
\label{defq}
& a_n=\frac {a_0}{p_1 p_2 \ldots p_n},\ n\geq 1\ ; \\
\nonumber & a_n=\frac
{a_1}{q_2^{n-1} q_3^{n-2} \ldots q_{n-1}^2 q_n} \left(
\frac{a_1}{a_0} \right) ^{n-1},\ n\geq 2.
\end{eqnarray}

Deciding whether a given entire function has only real zeros is a rather subtle problem. 
In 1926, J. I. Hutchinson found the following sufficient condition for an entire function 
with positive coefficients to have only real zeros.

{\bf Theorem B} (J. I. Hutchinson, \cite{hut}). { \it Let $f(z)=\sum_{k=0}^\infty a_k z^k$, $a_k > 0$ for all $k$. 
Then $q_n(f)\geq 4$, for all $n\geq 2,$  
if and only if the following two conditions are fulfilled:\\
(i) The zeros of $f(z)$ are all real, simple and negative, and \\
(ii) the zeros of any polynomial $\sum_{k=m}^n a_kz^k$, $m < n,$  formed 
by taking any number 
of consecutive terms of $f(z) $, are all real and non-positive.}

For some extensions of Hutchinson's results see,
for example, \cite[\S 4]{cc1}. 

We will use the well-known notion of a complex zero decreasing sequence. 
For a real polynomial $P $  we denote by $Z_c(P)$ the number
of nonreal zeros of $P$ counting multiplicities.

{\bf Definition 2}.  A sequence $(\gamma_k)_{k=0}^\infty$ of real
numbers is said to be a complex zero decreasing sequence (we write $(\gamma_k)_{k=0}^\infty 
\in \mathcal{CZDS}$), if
\begin{equation}
\label{czds}
 Z_c\left(\sum_{k=0}^n \gamma_k a_k z^k\right) \leq Z_c\left(\sum_{k=0}^n a_k z^k\right),
\end{equation}
for any real polynomial $\sum_{k=0}^n a_k z^k .$ 

The existence of nontrivial 
$\mathcal{CZDS}$ sequences is a consequence of the following remarkable 
theorem proved by Laguerre and extended by P\'olya.

{\bf Theorem C}  (G. P\'olya, see \cite{gpol} or
\cite[pp. 314-321]{pol}).   {\it Let $f$ be an entire function from the
 Laguerre-P\'olya class  having only negative zeros. Then
$\left(f(k)\right)_{k=0}^\infty \in \mathcal{CZDS}$.}

As it follows from the  theorem above,
\begin{equation}
\label{m1}
\left(a^{-k^2}\right)_{k=0}^\infty\in\mathcal{CZDS},\  a\geq 1,\quad
\left(\frac{1}{k!}\right)_{k=0}^\infty\in \mathcal{CZDS}.
\end{equation}

The entire function $g_a(z) =\sum _{j=0}^{\infty} z^j a^{-j^2}$, $a>1,$
a so called \textit{partial theta-function},  was investigated in the paper \cite{klv}. 
Simple calculations show that $q_n(g_a)=a^2$ for all $n.$ 
Since $\left(a^{-k^2}\right)_{k=0}^\infty\in\mathcal{CZDS}\ $ ,  for  $a\geq 1,$
we conclude that for every $n\geq 2$ there exists a constant $c_n >1$ such that  
$S_{n}(z,g_a):=\sum _{j=0}^{n} z^j a^{-j^2} \in \mathcal{L-P}$
$ \ \Leftrightarrow \ a^2 \geq c_n$.

The  survey \cite{War} by S.O. Warnaar contains the history of
investigation of the partial theta-function and its interesting properties.

{\bf Theorem D} (O. Katkova, T. Lobova, A. Vishnyakova, \cite{klv}).  {\it There exists a constant $q_\infty $
$(q_\infty\approx 3{.}23363666) $ such that:
\begin{enumerate}
\item
$g_a(z) \in \mathcal{L-P} \Leftrightarrow \ a^2\geq q_\infty ;$
\item
$g_a(z) \in \mathcal{L-P} \Leftrightarrow \ $  there exists $x_0 \in (- a^3, -a)$ such that $ \  g_a(x_0) \leq 0;$
\item
for a given $n\geq 2$ we have $S_{n}(z,g_a) \in \mathcal{L-P}$ $ \ \Leftrightarrow \ $
there exists $x_n \in (- a^3, -a)$ such that $ \ S_n(x_n,g_a) \leq 0;$
\item
$ 4 = c_2 > c_4 > c_6 > \ldots $  and    $\lim_{n\to\infty} c_{2n} = q_\infty ;$
\item
$ 3= c_3 < c_5 < c_7 < \ldots $  and    $\lim_{n\to\infty} c_{2n+1} = q_\infty .$

\end{enumerate}}

There is a series of works by V.P. Kostov dedicated to the interesting properties of zeros of the 
partial theta-function and its  derivative (see \cite{kos0}, \cite{kos1}, \cite{kos2}, \cite{kos3}, \cite{kos03},
\cite{kos04}, \cite{kos4}, \cite{kos5} and \cite{kos5.0}). For example, in  \cite{kos1}, V.P.~Kostov studied 
the so-called spectrum of the partial theta function, i.e. the set of values of 
$a>1$ for which the function  $g_a $ has a multiple real zero. 

{\bf Theorem E} (V.P. Kostov, \cite{kos1}).
{\it
\begin{enumerate}
\item
The spectrum $\Gamma$ of the partial theta-function consists of countably many values of $a$
denoted by $\widetilde{a}_1>\widetilde{a}_2 > \ldots > \widetilde{a}_k>\ldots  > 1,$
$\lim_{j \to \infty}\widetilde{a}_j = 1.$
\item
For $\widetilde{a}_k \in \Gamma$ the function $g_{\widetilde{a}_k}$ has exactly 
one multiple real zero which is of multiplicity 2 and is the rightmost of its real zeros.
\item
For $a \in (\widetilde{a}_{k+1}, \widetilde{a}_{k})$ the fucntion $ g_a$ has exactly $k$
complex conjugate pairs of zeros (counted with multiplicities). 
\end{enumerate}}

A wonderful paper \cite{kosshap} among the other results explains the role 
of the constant $q_\infty $  in the study of the set of entire functions with positive 
coefficients having all Taylor truncations with only real zeros. 

{\bf Theorem F} (V.P. Kostov, B. Shapiro, \cite{kosshap}).
{\it Let $f(z) = \sum_{k=0}^\infty a_k z^k$ be an entire function 
with positive coefficients and $S_n(z) = \sum_{j=0}^n a_j z^j$ be its sections.
Suppose that there exists  $ N\in {\mathbb{N}},$ such that for all  $  n \geq N$ the sections
$S_n(z) = \sum_{j=0}^n a_j z^j$ belong to the Laguerre-P\'olya class.
Then $\lim\inf_{n\to \infty} q_n(f) \geq q_\infty$.
}

In \cite{klv1}, some  entire functions  with a convergent sequence of second quotients 
of coefficients  are investigated. The main question of \cite{klv1} is whether a function 
and its Taylor sections belong to the Laguerre-P\'olya class. In \cite{BohVish} and 
\cite{Boh}, some important special functions with increasing sequence of second quotients 
of Taylor coefficients  are studied. 

In the previous paper \cite{ngthv1}, we  have studied the entire functions with positive Taylor coefficients such that 
$q_n(f)$ are decreasing in $n$.

{\bf Theorem  G} (T. H. Nguyen, A. Vishnyakova, \cite{ngthv1}).
{\it Let $f(z)=\sum_{k=0}^\infty a_k z^k $, $a_k > 0$ for all $k$, be an 
entire function.  Suppose that $q_n(f)$ are decreasing in $n$, i.e.  $q_2 \geq q_3 \geq 
q_4 \geq \ldots, $  and  $\lim\limits_{n\to \infty} q_n(f) = b \geq q_\infty$. 
Then all the zeros of $f$ are  real and negative, in other words  $f \in \mathcal{L-P}$.}

It is easy to see that, if only the estimation of $q_n(f)$ from below is given
and the assumption of monotonicity is omitted, then the constant $4$ in $q_n (f) \geq 4 $ 
is the smallest possible to conclude that $f \in \mathcal{L-P}$. 

In this paper, we study the case when $q_n(f)$ are increasing in $n$ and have obtained the following theorem.

\begin{Theorem}
\label{th:mthm1}
Let $f(z)=\sum_{k=0}^\infty a_k z^k $, $a_k > 0$ for all $k,$  be an 
entire function.  Suppose that the quotients $q_n(f)$ are increasing in $n$, 
and  $\lim\limits_{n\to \infty} q_n(f) = c < q_\infty$. 
Then the function $f$ does not belong to the  Laguerre-P\'olya class.
\end{Theorem}

\section{Proof of Theorem~\ref{th:mthm1}}

Without loss of generality, we can assume that $a_0=a_1=1,$ since we can 
consider a function $g(z) =a_0^{-1} f (a_0 a_1^{-1}z) $  instead of 
$f(z),$ due to the fact that such rescaling of $f$ preserves its property of 
having real zeros and preserves the second quotients:  $q_n(g) =q_n(f)$ 
for all $n.$ During the proof we use the notations $p_n$ and $q_n$ instead of 
$p_n(f)$ and $q_n(f).$ So we can write $ f(x) = 1 + x + \sum_{k=2}^\infty 
\frac{ x^k}{q_2^{k-1} q_3^{k-2} \ldots q_{k-1}^2 q_k}.$ We will also consider
a function  $\varphi(x) = f(-x) = 1 - x + \sum_{k=2}^\infty  \frac{ (-1)^k x^k}
{q_2^{k-1} q_3^{k-2} \ldots q_{k-1}^2 q_k}$  instead of $f.$

Since the quotients $q_n$ are increasing in $n$, 
and  $\lim\limits_{n\to \infty} q_n = c < q_\infty,$ we conclude 
that $q_2 \leq q_\infty <4.$ The following lemma shows that 
for $q_2  <3$ we have  $\varphi \notin \mathcal{L-P}.$

\begin{Lemma}
\label{th:lm5}
Let $\varphi(z) = \sum_{k=0}^\infty (-1)^k a_k z^k$ be an entire function, 
$a_k > 0$ for all $k,$ $a_0 = a_1 = 1$,  and $q_n =q_n(\varphi)$ are increasing in $n$, i.e.  
$q_2 \leq q_3 \leq q_4 \leq \ldots . $  
If $\varphi \in \mathcal{L-P},$ then  $q_2(f) \geq 3$.
\end{Lemma}

\begin{proof}
Denote by  $0 < z_1 \leq z_2 \leq z_3 \leq  \ldots $  the real roots of $\varphi $. 
We observe that
$$0 \leq \sum_{k=1}^\infty \frac{1}{z_k^2} = \left( \sum_{k=1}^\infty \frac{1}{z_k}  \right)^2 - 
2 \sum_{1\leq i< j < \infty} \frac{1}{z_i z_j} = \left(\frac{a_1}{a_0}\right)^2 - 2\frac{a_2}{a_0},$$
whence  $q_2 \geq 2.  $

According to the Cauchy-Bunyakovsky-Schwarz inequality, we obtain
$$(\frac{1}{z_1} + \frac{1}{z_2} + ...)(\frac{1}{z_1^3} + \frac{1}{z_2^3} + 
\ldots) \geq (\frac{1}{z_1^2 }+ \frac{1}{z_2^2} + \ldots)^2.$$

By Vieta's formulas, we have $ \sigma_1:= \sum_{k=1}^\infty \frac{1}{z_k} =  \frac{a_1}{a_0},$ 
$\sigma_2 = \sum_{1<i<j<\infty} \frac{1}{z_iz_j} = \frac{a_2}{a_0},$ and
$\sigma_3 = \sum_{1<i<j<k<\infty} \frac{1}{z_iz_jz_k} =  \frac{a_3}{a_0}.$ 
Further, we need the following identities:
$\sum_{k=1}^\infty \frac{1}{z_1^2} = \sigma_1^2 - 2\sigma_2,$ and 
$\sum_{k=1}^\infty \frac{1}{z_1^3} = \sigma_1^3 - 3\sigma_1\sigma_2 + 3\sigma_3.$ 
Consequently, we have 
$$\sigma_1 (\sigma_1^3 - 3\sigma_1\sigma_2 + 3\sigma_3) \geq (\sigma_1^2 - 2\sigma_2)^2,$$
or
$$\frac{a_1^2a_2}{a_0^3} + 3\frac{a_1a_3}{a_0^2} - 4\frac{a_2^2}{a_0^2} \geq 0.  $$

Since $a_0 = a_1 = 1$ and $a_2 = \frac{1}{q_2}, a_3 = \frac{1}{q_2^2q_3}$, we get: 
$$q_3(q_2 - 4) + 3 \geq 0.$$
Since we have the conditions that $q_2 < 4$ and $q_2 \leq q_3,$ we conclude that
$$q_2(q_2 - 4) + 3 \geq 0.$$

Thus, we get that $q_2 \geq 3.$
\end{proof}

Further, we assume that $3 \leq q_2 < q_\infty.$ 

In order to prove Theorem~\ref{th:mthm1}, we need some more Lemmas.

\begin{Lemma}
\label{th:lm1}
Let $\varphi(x) =  1 - x + \sum_{k=2}^\infty  \frac{ (-1)^k x^k}
{q_2^{k-1} q_3^{k-2} \ldots q_{k-1}^2 q_k} $  be an entire function.
Suppose that   $q_2 \geq 2,$  $q_k$ 
are increasing in $k$,   i.e. $q_2 \leq q_3 \leq q_4 \ldots$, 
and $\lim\limits_{n \to \infty} q_n = 
c < q_\infty$.  Then for any $x \in [0, q_2]$ we have  $\varphi(x)>0$, i.e. 
there are no real roots of $\varphi$ in the segment $[0; q_2]$.
\end{Lemma}

\begin{proof}  For $x \in [0;1]$ we have 
$$ 1 \geq x  > \frac{x^2}{q_2} >  \frac{x^3}{q_2^2 q_3}  > 
\frac{x^4}{q_2^3 q_3^2 q_4 } > \cdots  , $$
whence 
\begin{equation}
\label{mthm1.1}
 \varphi(x) >0 \quad \mbox{for all}\quad x\in [0;1].
\end{equation}

Suppose that   $x \in (1; q_2].$  Then we obtain 
\begin{equation}
\label{mthm1.2}  1 <  x  \geq  \frac{x^2}{q_2} > \frac{x^3}{q_2^2 q_3} > 
\cdots >  \frac{  x^k}{q_2^{k-1} q_3^{k-2} \ldots q_{k-1}^2 q_k} > \cdots 
\end{equation}

For an arbitrary  $m \in {\mathbb{N}}$  we have
$$\varphi(x) = \left(1 - x + \sum_{k=2}^{2m+1}  \frac{ (-1)^k x^k}
{q_2^{k-1} q_3^{k-2} \ldots q_{k-1}^2 q_k}\right) $$
$$ +  \sum_{k=2m+2}^\infty  \frac{ (-1)^k x^k}
{q_2^{k-1} q_3^{k-2} \ldots q_{k-1}^2 q_k} =: S_{2m+1}(x, \varphi) + 
R_{2m+2}(x, \varphi).$$
By (\ref{mthm1.2}) we obtain $R_{2m+2}(x, \varphi) >0$ for all   $x \in (1; q_2],$
or
\begin{equation}
\label{mthm1.3}   \varphi(x) > S_{2m+1}(x, \varphi)\quad \mbox{for all} \quad 
x \in (1; q_2], m \in {\mathbb{N}}.
\end{equation} 
It remains to prove that there exists  $m \in {\mathbb{N}}$ such that 
$S_{2m+1}(x, \varphi) >0 $ for all  $x \in (1; q_2].$ We have
\begin{eqnarray}
\label{m5}
&S_{2m+1}(x, \varphi) = (1 - x) + \left(\frac{x^2}{q_2} 
- \frac{x^3}{q_2^2 q_3}\right)
+ \left(\frac{x^4}{q_2^3 q_3^2 q_4} - \frac{x^5}{q_2^4 q_3^3 q_4^2 q_5}\right)+ \ldots \\
\nonumber & + 
\left(\frac{x^{2m}}{q_2^{2m-1} q_3^{2m-2} \cdot \ldots \cdot 
q_{2m-1}^2 q_{2m}} - \frac{x^{2m+1}}{q_2^{2m} q_3^{2m-1} \cdot \ldots \cdot 
q_{2m}^2 q_{2m+1}}\right).
\end{eqnarray}

Under our assumptions,  $q_k$ 
are increasing in $k$,   and $\lim\limits_{n \to \infty} q_n = 
c .$  We prove that for any fixed $k = 1, 2, \ldots, m$ and $x \in (1; q_2]$  the following inequality holds:
$$\frac{x^{2k}}{q_2^{2k-1} q_3^{2k-2}\cdot \ldots \cdot q_{2k}} -
 \frac{x^{2k+1}}{q_2^{2k} q_3^{2k-1}\cdot \ldots \cdot q_{2k}^2 
q_{2k+1}}  \geq $$
 $$\frac{x^{2k}}{c^{2k-1}\cdot c^{2k-2}\cdot \ldots \cdot c} - \frac{x^{2k+1}}{c^{2k} 
\cdot c^{2k-1}\cdot \ldots \cdot c^2 \cdot c}.$$

For  $x \in (1; q_2]$ and $k=1, 2, \ldots , m,$ we define the following function
$$F(q_2, q_3, \ldots, q_{2k}, q_{2k+1}):=\frac{x^{2k}}{q_2^{2k-1} q_3^{2k-2}\cdot \ldots \cdot q_{2k}} -
 \frac{x^{2k+1}}{q_2^{2k} q_3^{2k-1}\cdot \ldots \cdot q_{2k}^2 
q_{2k+1}}.$$ 
 We can observe that
$$\frac{\partial F(q_2, q_3, \ldots, q_{2k}, q_{2k+1}) }
 {\partial q_2} = - \frac{(2k-1) x^{2k}}{q_2^{2k} q_3^{2k-2} \ldots q_{2k}} + 
 \frac{2k x^{2k+1}}{q_2^{2k+1}q_3^{2k-1} \ldots q_{2k}^2q_{2k+1}} < 0$$ $$\Leftrightarrow x < 
 \Big(1 - \frac{1}{2k} \Big) q_2 q_3 \ldots q_{2k}q_{2k+1}.$$ 
 
Thus, under our assumptions, the function $F(q_2, q_3, \ldots, q_{2k}, q_{2k+1})$ 
is decreasing in $q_2$.  Since $q_2 \leq q_3,$ 
$$F(q_2, q_3, q_4, \ldots, q_{2k}, q_{2k+1}) \geq F(q_3, q_3, q_4,  \ldots, q_{2k}, q_{2k+1})=$$
$$\frac{x^{2k}}{q_3^{4k-3} q_4^{2k-3} \ldots q_{2k}} - 
\frac{x^{2k+1}}{q_3^{4k-1}q_4^{2k-2} \ldots q_{2k+1}}.$$
Further we have 
$$\frac{\partial F(q_3, q_3, q_4,  \ldots, q_{2k}, q_{2k+1})}
 {\partial q_3} = - \frac{(4k-3) x^{2k}}{q_3^{4k-2} q_4^{2k-3} \ldots q _{2k}} 
 + \frac{(4k-1) x^{2k+1}}{q_3^{4k}q_4^{2k-2} \ldots q_{2k+1}} < 0$$ 
 $$\Leftrightarrow x < \frac{4k-3}{4k-1}q_3^2 q_4 \ldots q_{2k+1}. $$
 
Thus, under our assumptions, $F(q_3, q_3, q_4, \ldots, q_{2k}, q_{2k+1})$ is decreasing in $q_3$, and, since, $q_3 \leq q_4$
we obtain
 $$F(q_3, q_3, q_4 \ldots, q_{2k}, q_{2k+1}) \geq F(q_4, q_4, q_4, q_5,  \ldots, q_{2k}, q_{2k+1}).$$
 
Analogously, we obtain the following chain of inequalities
$$F(q_2, q_3, q_4, \ldots, q_{2k}, q_{2k+1}) \geq F(q_3, q_3, q_4,  \ldots, q_{2k}, q_{2k+1}) \geq $$  
$$ F(q_4, q_4, q_4, q_5,  \ldots, q_{2k}, q_{2k+1}) \geq \ldots  
 \geq F(q_{2k+1}, q_{2k+1}, \ldots,  q_{2k+1}, q_{2k+1}). $$
 
Further, we have
 
$$\frac{\partial F(q_{2k+1}, q_{2k+1}, \ldots,  q_{2k+1}, q_{2k+1}) }
{\partial q_{2k+1}}  = - \frac{(2k^2-k) x^{2k}}{q_{2k+1}^{2k^2-k+1}} + 
\frac{(2k^2+k) x^{2k+1}}{q_{2k+1}^{2k^2+k+1}} < 0$$
$$ \Leftrightarrow x < \frac{2k^2-k}{2k^2+k} q_{2k+1}^{2k}.$$
 
Thus, $F(q_{2k+1}, q_{2k+1}, \ldots,  q_{2k+1}, q_{2k+1})$ is decreasing in 
$q_{2k+1}$, and since $q_k$ 
are increasing in $k$,   and $\lim\limits_{n \to \infty} q_n = 
c ,$ we conclude that 
$$F(q_{2k+1}, q_{2k+1}, \ldots,  q_{2k+1}, q_{2k+1}) \geq   F(c, c, \ldots,  c, c) = $$
$$ \frac{x^{2k}}{ c^{k(2k-1)}} -
 \frac{x^{2k+1}}{c^{k(2k+1)}}.$$
 Substituting the last inequality in (\ref{m5}) for every  $x \in (1; q_2]$ and $k=1, 2, \ldots , m,$ , we get
\begin{eqnarray}
\label{mthm1.4}
& S_{2m+1}(x, \varphi) \geq (1 - x) + \left(\frac{x^2}{c} 
- \frac{x^3}{c^3}\right)
+ \left(\frac{x^4}{c^6} - \frac{x^5}{c^{10}}\right)+ \ldots \\
\nonumber & + \left(\frac{x^{2m}}{c^{m(2m-1)}} - \frac{x^{2m+1}}{c^{m(2m+1)} }\right) 
= \sum_{k=0}^{2m+1} \frac{x^k}{\sqrt{c}^{k(k-1)}} = S_{2m+1}( \sqrt{c} x, g_{\sqrt{c}}),
\end{eqnarray}
where $g_a$ is the partial theta-function and $S_{2m+1}(y, g_a)$ are its $(2m+1)$-th section
at the point  $y$. Since by our assumptions  $(\sqrt{c})^2 < q_\infty,$ using the statement (5) of Theorem D 
we obtain that there exists  $m \in {\mathbb{N}}$ such that $S_{2m+1}( y, g_{\sqrt{c}})\notin \mathcal{L-P}.$
Let us choose and fix such $m.$ By the statement (3) of Theorem D  we obtain that  for every $x$ such that
$\sqrt{c} < \sqrt{c} x < (\sqrt{c})^3$ we have $S_{2m+1}( \sqrt{c} x, g_{\sqrt{c}})>0.$ It means that 
for every $x: 1< x < c$ we have  $S_{2m+1}( \sqrt{c} x, g_{\sqrt{c}})>0, $ and, using (\ref{mthm1.4}) 
and (\ref{mthm1.3}),
$$  \varphi(x) > S_{2m+1}(x, \varphi) >0 \quad \mbox{for all} \quad  x \in (1; q_2)\subset (1; c). $$
It remains to prove that $ \varphi(q_2) >0.$ We have 
$$   \varphi(q_2) = \left(1 - q_2 + q_2 - \frac{q_2}{q_3}\right) + \left(\frac{q_2}{q_3^2 q_4} 
- \frac{q_2}{q_3^3 q_4^2 q_5}\right) $$
$$ +\left( \frac{q_2}{q_3^4 q_4^3 q_5^2 q_6} - \frac{q_2}{q_3^5 q_4^4 q_5^3 q_6^2 q_7} \right) +\ldots >0 $$
by our assumptions on $q_j.$
\end{proof}

\begin{Lemma}
\label{th:lm2}
Let $P(z) = 1 - z + \frac{z^2}{a} - \frac{z^3}{a^2b} + \frac{z^4}{a^3b^2c}$ be a polynomial, 
$3 \leq  a < 4,$ and $a \leq b \leq c.$ 
Then $$\min_{0 \leq \theta \leq 2\pi}|P(ae^{i \theta})| \geq \frac{a}{b^2c}.$$
\end{Lemma}

\begin{proof} By the direct calculation we have
$$|P(ae^{i\theta})|^2 =
 (1 - a\cos{\theta} + a\cos{2\theta} - \frac{a}{b}\cos{3\theta} + \frac{a}{b^2c}\cos{4\theta})^2 +$$
$$ (- a\sin{\theta} + a\sin{2\theta} - \frac{a}{b}\sin{3\theta} + \frac{a}{b^2c}\sin{4\theta})^2$$
$$=1 + 2a^2 + \frac{a^2}{b^2} + \frac{a^2}{b^4c^2} - 2a\cos{\theta} + 2a\cos{2\theta} - 2\frac{a}{b}\cos{3\theta}$$
$$+2\frac{a}{b^2c}\cos{4\theta} - 2a^2\cos{\theta} + 2\frac{a^2}{b}\cos{2\theta} - 2\frac{a^2}{b^2c}\cos{3\theta}$$
$$- 2\frac{a^2}{b}\cos{\theta} + 2\frac{a^2}{b^2c}\cos{2\theta} - 2\frac{a^2}{b^3c}\cos{\theta}.$$

Denote by $t := \cos{\theta}, t \in [-1;1].$ Since $\cos{2\theta} =  2t^2 - 1,$
$\cos{3\theta} = 4t^3 - 3t,$ and $\cos{4\theta} =   8t^4 - 8t^2 +1,$ we get

$$|P(ae^{i\theta})|^2 = \frac{16a}{b^2c}t^4 + \left(-\frac{8a}{b} - \frac{8a^2}{b^2c}\right)t^3 + 
\left(4a - \frac{16a}{b^2c} + 
\frac{4a^2}{b} + \frac{4a^2}{b^2c}\right)t^2 +$$
$$ \left(-2a + \frac{6a}{b} - 2a^2 + \frac{6a^2}{b^2c} - \frac{2a^2}{b} - \frac{2a^2}{b^3c}\right)t $$
$$ + \left(1 + 2a^2 + \frac{a^2}{b^2} + \frac{a^2}{b^4c^2} - 2a + \frac{2a}{b^2c} - 
\frac{2a^2}{b} - \frac{2a^2}{b^2c}\right).$$

We want to show that $\min_{0 \leq \theta \leq 2\pi} |P(ae^{i\theta})|^2 \geq \frac{a^2}{b^4c^2}.$
Using the last expression we see that the inequality we want to get is equivalent to the following: for
 all  $t \in [-1; 1]$
the next inequality holds
$$\frac{16a}{b^2c}t^4 - \frac{8a}{b} \big( 1+ \frac{a}{bc} \big)t^3 + 4a \big( 1 - \frac{4}{b^2c} + 
\frac{a}{b} + \frac{a}{b^2c}\big)t^2 - 2a\big( 1 - \frac{3}{b} +a - \frac{3a}{b^2c} + \frac{a}{b} $$
$$+ \frac{a}{b^3c} \big)t + \big( 1 + 2a^2 + \frac{a^2}{b^2} - 2a + \frac{2a}{b^2c} - \frac{2a^2}{b} - 
\frac{2a^2}{b^2c} \big) \geq 0.$$
Let $y := 2t,$  $y \in [-2; 2].$ We rewrite the last inequality in the form
$$\frac{a}{b^2c}y^4 - \frac{a}{b} \left( 1+ \frac{a}{bc} \right)y^3 + a \left( 1 - 
\frac{4}{b^2c} + \frac{a}{b} + \frac{a}{b^2c}\right)y^2 $$
$$- a\left( 1 - \frac{3}{b} +a - \frac{3a}{b^2c} + \frac{a}{b} + \frac{a}{b^3c} \right)y + $$
$$\left( 1 + 2a^2 + \frac{a^2}{b^2} - 2a + \frac{2a}{b^2c} - \frac{2a^2}{b} - \frac{2a^2}{b^2c} \right) \geq 0.$$
 
We note that the coefficient of $y^4$ is positive, and the coefficient of  $y^3$ is negative.  It is easy to show that the 
other coefficients are also sign-changing. For $y^2$:  $1 - \frac{4}{b^2c} > 0$ since $b^2c > 4$, 
thus, $1 + \frac{a}{b} + \frac{a}{b^2c} - \frac{4}{b^2c} = (1 - \frac{4}{b^2c}) + \frac{a}{b} + \frac{a}{b^2c} > 0.$
For $y$:   $1 + a + \frac{a}{b} + \frac{a}{b^3c} - \frac{3}{b} - \frac{3a}{b^2c} = (1 + a - \frac{3}{b}) + $
$(\frac{a}{b} - \frac{3a}{b^2c}) + \frac{a}{b^3c} > 0.$   Finally, $1 + 2a^2 + \frac{a^2}{b^2} - 2a - 2\frac{a^2}{b} - 
2\frac{a^2}{b^2c} + 2\frac{a}{b^2c} =$
$(1 + a^2 - 2a) + (a^2 - 2\frac{a^2}{b}) + (\frac{a^2}{b^2} - 2\frac{a^2}{b^2c}) + 2\frac{a}{b^2c} > 0$ since 
$1 - 2a + a^2  \geq 0;$  $a^2 - 2\frac{a^2}{b} > 0 $  and $\frac{a^2}{b^2} - 2\frac{a^2}{b^2c} > 0$ by
our assumptions.
 
Consequently, the inequality we need holds for any $y \in [-2; 0]$, and we have to prove it for $y \in [0; 2]$.
Multiplying our inequality by $\frac{b^2c}{a}$, we get 
$$y^4 - (bc + a)y^3 + (b^2c + abc + a - 4) y^2 - (b^2c + ab^2c + abc + 
\frac{a}{b} - 3bc - 3a)y$$
$$ + (\frac{b^2c}{a} + 2ab^2c + ac - 2b^2c - 2abc - 2a +2) =: \psi(y),$$ 
and we want to prove 
that $\psi (y) \geq 0$ for all $y \in [0; 2].$
 
Let $\chi(y) := \psi(y) - \frac{1}{b}(b - a) y,$ whence $\chi(y) \leq \psi(y)$ for all $y \in [0; 2]$.
It is sufficient to prove that $\chi (y) \geq 0$ for all $y \in [0; 2].$ We have  $\chi(0) = \psi(0) = 
\frac{b^2c}{a} + 2ab^2c + ac - 2b^2c - 2abc - 2a + 2 \geq 0,$ as it was previously shown.
We also have  $\chi(2) = \psi(2) - \frac{2}{c}(b - a) \geq 0$ since
$$\psi(2) = -2bc - 2\frac{a}{b} + \frac{b^2c}{a} + ac + 2 =$$
$$\frac{1}{b} \left(2(b - a) + \frac{b^2c}{a}(b - a) - bc(b - a)\right) =$$
$$\frac{1}{b}(b - a) \left(2 + \frac{bc}{a} (b - a)\right) \geq \frac{2}{b} (b - a) \geq 0.$$

Now we consider the following function:  $\nu(y):=\frac{\partial^2 \chi(y)}{\partial y^2} = 
\frac{\partial^2 \psi(y)}{\partial y^2} = 12 y^2 - 6 (bc + a)y + 2(b^2c + abc + a - 4) .$
The vertex point of this parabola is $y_v = \frac{bc + a}{4}  \geq 3.$  
Thus, we can observe that $\nu(y)$ decreases for $y \in [0; 2].$ We have 
$\nu(0) = 2(b^2c + abc + a - 4) > 0,$ and $\nu(2) = 2abc + 2b^2c - 12bc - 10a + 40.$
We want to show that $\nu(2)$ is positive. We have
$$abc + b^2c - 6bc - 5a + 20 = (20 - 5a) + (b^2c - 3bc) + (abc - 3bc) = $$
$$5(4 - a) + bc(c - 3) + bc(a - 3) > 0$$ 
due to our assumptions. We  conclude that $\nu(y)$ is nonnegative for $y \in [0; 2]$, 
and it follows that $\chi^\prime(y)$ increases for $y \in [0; 2].$

We want to show that $\chi^\prime(y) \leq 0$ for $y \in [0; 2]$, and it is sufficient to show 
that $\chi^\prime(2) \leq 0.$ We have
$$\chi^\prime(2) =  \psi^\prime(2) - \frac{b-a}{b} = 15 - 9bc - 5a + 3b^2c + 3abc - ab^2c  =$$
$$ 5(3 - a) + bc(-9 + 3b + 3a - ab) = 5 (3 - a) + bc(a - 3)(3 - b) \leq 0.$$

Thus, $\chi(y)$ decreases, $\chi(2) \geq 0,$ so it is positive for $y \in [0; 2]$. Since 
$\chi(y) \leq \psi(y)$, it follows that 
$\psi(y)$ is positive for $y \in [0; 2].$
\end{proof}

The function $ \varphi$ can be presented in the following form: 
$$ \varphi(x) = (1 - x + \frac{x^2}{q_2} - \frac{x^3}{q_2^2 q_3} + \frac{x^4}{q_2^3q_3^2 q_4}) + $$
$$\sum_{k=5}^\infty \frac{(-1)^k x^k}{q_2^{k-1} q_3^{k-2} \ldots q_k} =: S_4(x,  \varphi) + R_5(x,  \varphi).$$
By Lemma \ref{th:lm2} we have 
\begin{equation}
\label{estbel}
\min_{0 \leq \theta \leq 2\pi}|S_4 (q_2 e^{i \theta},  \varphi)| \geq \frac{q_2}{q_3^2q_4}.
\end{equation}
Now we need the estimation on $|R_5 (q_2 e^{i \theta},  \varphi)|$ from above.

\begin{Lemma}
\label{th:lm3}
Let $R_5(z,  \varphi):= \sum_{k=5}^\infty \frac{(-1)^k z^k}{q_2^{k-1} q_3^{k-2} \ldots q_k}$,  
$q_n $ are increasing in $n$,
and $\lim\limits_{n \to \infty} q_n(f) = c < q_\infty$.
Then 
$$\max_{0 \leq \theta \leq 2\pi}|R_5(q_2e^{i\theta},  \varphi)| \leq \frac{q_2}{q_3^3 q_4^3 - q_3^2}.$$ 
\end{Lemma}

\begin{proof}
We have 

$$|R(q_2e^{i\theta},  \varphi)| \leq \sum_{k=5}^\infty \frac{q_2^k}{q_2^{k-1} q_3^{k-2} \ldots q_k} = 
\sum_{k=5}^\infty \frac{q_2}{q_3^{k-2} \ldots q_k} =$$ $$\frac{q_2}{q_3^3 q_4^2 q_5} + 
\frac{q_2}{q_3^4 q_4^3 q_5^2 q_6} + \ldots + \frac{q_2}{q_3^{k-2}\ldots q_k} + \ldots $$ 
$$\leq \frac{q_2}{q_3^3 q_4^3} (1 + \frac{1}{q_3 q_4^3} + \frac{1}{q_3^2 q_4^7} + \ldots + 
\frac{1}{q_3^{k-2} q_4^{\frac{(k-2)(k-3)}{2}}} + \ldots) =$$
$$ \frac{q_2}{q_3^3 q_4^3} \cdot \frac{1}{1 - \frac{1}{q_3 q_4^3}} = 
\frac{q_2}{q_3^3 q_4^3 - q_3^2}.  $$
\end{proof}

Let us check that  $\frac{q_2}{q_3^2q_4}  > \frac{q_2}{q_3^3 q_4^3 - q_3^2},$
which is equivalent to $q_4 < q_3q_4^3-1.$ The last inequality obviously holds under 
our assumptions.  Consequently, according to Rouch\'e's theorem, the functions $S_4(z,  \varphi)$ 
and $\varphi (z)$ have the same number of zeros inside the circle $\{ z: |z| < q_2   \}$ counting 
multiplicities. 

It remains to prove that $S_4(z,  \varphi)$  has  zeros in the circle $\{ z: |z| < q_2   \}$.
To do this we need the notion of apolar polynomials and the famous theorem by J.H. Grace.

\begin{Definition} (see, for example \cite[Chapter 2, \$ 3, p. 59]{PS}).
\label{th:r1} Two complex polynomials $P(z) = \sum_{k=0}^n {n \choose k}a_k z^k$ and  
$Q(z) = \sum_{k=0}^n {n \choose k}b_k z^k$ of degree $n$ are called apolar if 
\begin{equation}
\label{apol}
\sum_{k=0}^n (-1)^k {n \choose k} a_k b_{n-k} = 0.
\end{equation}
\end{Definition} 
 
The following famous theorem due to J.H. Grace states that the complex zeros of two apolar polynomials 
cannot be separated by a straight line or by a circle.

{\bf Theorem H} (J.H. Grace, see for example \cite[Chapter 2, \$ 3, Problem 145]{PS}). 
Suppose $P$ and $Q$ are two apolar polynomials 
of degree $n \geq 1.$ If all zeros of $P$ lie in a circular region $C,$ then $Q$ has at least one zero in $C.$ 
(A circular region is a closed or open half-plane, disk or exterior of a disk).

\begin{Lemma} 
\label{th:lm4}
Let $S_4(z,  \varphi) = 1 - z + \frac{1}{q_2}z^2 - \frac{1}{q_2^2q_3}z^3 + 
\frac{1}{q_2^3q_3^2q_4}z^4$ be a polynomial 
and $q_2 \geq 3$. Then $S_4(z,  \varphi)$ has at least one root in the circle $\{z:|z| \leq q_2\}$.
\end{Lemma}

\begin{proof}
We have $S_4(z,  \varphi) = {4 \choose 0} + {4 \choose 1} (-\frac{1}{4})z + {4 \choose 2} \frac{1}{6q_2}z^2 + 
{4 \choose 3} (-\frac{1}{4q_2^2 q_3})z^3 + {4 \choose 4} \frac{1}{q_2^3q_3^2q_4}z^4$.  Let $Q(z) = 
{4 \choose 2}b_2z^2 + {4 \choose 3}b_3z^3 + {4 \choose 4}z^4.$ Then the condition for $S_4(z,  \varphi)$ 
and $Q(z)$ to be apolar is the following
$${4 \choose 0} - {4 \choose 1}\left(-\frac{1}{4}\right)b_3 + {4 \choose 2}\frac{1}{6q_2}b_2 = 0.$$

We have $1 - b_3 + \frac{b_2}{q_2} = 0.$  Further we choose 
$b_3 = \frac{q_2 - 6}{2}$, and, by the apolarity condition, $b_2 = - q_2(1+\frac{q_2 - 6}{2}).$
So we have 
$$Q(z) = - 6 q_2 \left(1 + \frac{q_2 - 6}{2}\right)z^2 + 4 \left(\frac{q_2 - 6}{2}\right) z^3 + z^4 $$
$$ = z^2 \left( - 3q_2 (q_2 - 4) + 2(q_2-6)z + z^2\right).$$

As we can see,  the zeros  of $Q$ are $z_{1} = 0, z_{2} = 0, z_{3} = q_2, z_{4} = - 3(q_2 - 4).$
To show that $z_4$ lies in the circle of radius $q_2,$  we solve the inequality  $|-3(q_2 - 4)| \leq q_2$. 
Thus, we obtain that if $q_2 \geq 3, $ then all zeros of $Q$ are in the circle $\{z: |z| \leq q_2  \}.$
Since all the zeros of $Q$ are  in the circle $\{z: |z| \leq q_2  \},$ we obtain by the Grace theorem that 
$S_4(z,  \varphi)$ has at least one zero in the circle $\{z:|z| \leq q_2\}.$
\end{proof}

Thus,  $S_4(z,  \varphi)$ has at least one zero in the circle $\{z:|z| \leq q_2\}, $ and, applying Lemma
\ref{th:lm2} to the $S_4(z,  \varphi)$, we conclude that $S_4(z,  \varphi)$ does not have zeros on 
$\{z:|z| = q_2  \}.$ So, the polynomial $S_4(z,  \varphi)$ has at least one zero in the open 
circle $\{z:|z| < q_2\}.$ By the Rouch\'e's theorem, the functions $S_4(z,  \varphi)$ 
and $\varphi (z)$ have the same number of zeros inside the circle $\{ z: |z| < q_2   \},$ 
whence $\varphi$ has at least one zero in the open circle $\{z:|z| < q_2\}.$ If $\varphi $
is in the Laguerre-P\'olya class, this zero must be real, and, since coefficients of $\varphi $
are sign-changing, this zero belongs to $(0; q_2).$ However, by Lemma \ref{th:lm1}, we have
$\varphi (x) >0$ for all $x\in [0; q_2].$ This contradiction shows that $\varphi  \notin \mathcal{L-P}.$

Theorem~\ref{th:mthm1} is proved.

\end{document}